\documentclass[11pt]{article}
\usepackage{amssymb}
\usepackage{enumerate}
\usepackage{array}
\usepackage{longtable}
\usepackage{textcomp}
\usepackage{latexsym}
\usepackage{varioref}
\usepackage{xspace}
\usepackage{makeidx}
\usepackage{verbatim}
\usepackage{graphicx}
\usepackage{tabularx}
\usepackage{amsfonts}
\usepackage{amsmath}
\usepackage{amsthm}
\usepackage{amscd}
\usepackage{graphicx}
\usepackage{hyperref}
\usepackage{float}
\usepackage[left=2cm,right=2cm,top=2cm,bottom=2cm,bindingoffset=0cm]{geometry}

\newcommand*\diff{\mathop{}\!\mathrm{d}}
\newtheorem{Theorem}{Theorem}
\newtheorem*{remark}{Remark}

\begin{document}

\date{}
\author{Arseny Raiko}
\title{Comparison of the first positive Neumann eigenvalues for rectangles and special parallelograms}
\maketitle{}

\begin{abstract}
First non-zero Neumann eigenvalues of a rectangle and a parallelogram with the same base and area are compared in case when the height of the parallelogram 
is greater than the base. This result is applied to compare first non-zero Neumann eigenvalue normalized by the square of the perimeter on the parallelograms
with a geometrical restriction and the square. The result is inspired by Wallace--Bolyai--Gerwien theorem. An interesting three-dimensional problem related
to this theorem is proposed.
\end{abstract}

Let $\Omega$ be a planar simply connected bounded domain. Suppose $\partial \Omega$ satisfies all the necessary regularity assumptions
 needed to gurantee that the Neumann spectrum of $\Omega$ is discrete. Denote by $\mu_n$ Neumann eigenvalues of the Laplacian on the domain,
by $p(\Omega)$ the boundary volume of the domain (i. e. the perimeter) and by $d(\Omega)$ the diameter of $\Omega$. 
The first Neumann eigenvalue $\mu_1$ is always $0$, because constant functions satisfy Neumann boundary conditions.
So the first non-zero Neumann eigenvalue is $\mu_2$.
Consider the variational problem $$\sup \{ \mu_2(\Omega): \Omega ~\text{ is open bounded domain with} ~p(\Omega) = c \}.$$
As it is explained in \cite{Bucur}, this maximization problem is ill-posed. The supremum is unbounded on the sequence of squares with a side replaced by a sinusoidal arc.
Remark that this is a sequence of non-convex domains. However, this problem is well posed if we restrict our consideration to convex domains $\Omega$.
In the paper \cite{Kroger} of Kr\"oger it is shown that $\mu_2(\Omega)d^2(\Omega)$ is bounded above in the convex case.
And also it is well known that in convex case $\dfrac{p(\Omega)}{d(\Omega)} \leq \pi$ (for instance, see \cite{Cifre}).
 So, $\mu_2(\Omega)$ is bounded above for the convex domains with the fixed perimeter.
One can replace the variational problem by the problem
of finding the supremum of the functional $M_2(\Omega) = \mu_2(\Omega) p(\Omega)^2$ on planar domains. This is so-called normalized eigenvalue.
Multiplication by the square of the perimeter is necessary to make the functional homothety invariant. There is a long-term stated conjecture that
the maximal value of this functional on convex domains is attained both on the equilateral triangle ($E$) and the square ($S$).
It's well-known that $M_2(E) = M_2(S) = 16 \pi^2$ (Remark that the second Neumann eigenvalue has multiplicity 2 on both domains).
This problem could be interpreted as an analogue of Szeg\H{o}'s inequality with another geometrical normalization,
the perimeter of the border instead of the volume of the domain.
Original Szeg\H{o}'s inequality states that $\mu_2(\Omega) Area(\Omega) \leq \mu_2(D) Area(D)$, where $D$ is a disk
 and the equality obtained only when $\Omega$ is a disk \cite{Szego}. 
Normalization by the perimeter makes the problem more complicated, than the normalization by the area.
So it becomes interesting to show that $M_2(\Omega) \leq M_2(S)$ holds at least in the polygonal class of domains $\Omega$. 
Laugesen and Siudeja have shown in \cite{Siudeja} that $M_2(\Omega) \leq M_2(E)$ holds in the class of all triangles.
Then one can try to show that $M_2(\Omega) \leq M_2(E)$ holds in the class of convex quadraliterals. 
 This paper provides this inequality in a class of parallelograms with a geometrical restriction.

\begin{Theorem}{}
Consider a parallelogram $Q$ with the lenghts of the sides $a \geq b$ and the smallest angle~$\alpha$.
Let $R$ be a rectangle with the same volume and the same shorter side as $Q$. 
If $\dfrac{b}{a} \leq sin \alpha$, then

\begin{equation}
\label{QlR}
\mu_2(Q) \le \mu_2(R).
\end{equation}

\end{Theorem}

\begin{proof}

First, let's show that $(\ref{QlR})$  holds in a special case. Consider the case of parallelograms with $a \cos \alpha\leq~b$.
This means that the projection of the longer side onto the line containing the shorter side is not greater than the shorter side.
This restriction is not necessary but in this case the proposed construction is simpler.

Define the coordinate system, associated with the given parallelogram $ABCD$ such that
 $x$ axis runs along the shorter side $AB$, $y$ axis is orthogonal to $x$ and all vertices of parallelogram have non-negative $y$ coordinates
(see \autoref{fig:coorsys}). Let us denote the intersection of the line $DC$ with the $y$ axis by $F$. Choose $E$ on the line $DC$ such that
 $ABEF$ is a rectangle. It is easy to show that $ABEF$ has the same base and volume as $ABCD$.

The condition $\dfrac{b}{a}  \leq \sin \alpha$ guarantees that $|BE| = a \sin \alpha \geq b = |AB|$, i.e the height $BE$ is greater than the side $AB$.
The additional condition stating that the projection of the longer side onto the shorter side is not greater than the shorter side provides
 that the point $E$ lies between $D$ and $C$.

\begin{figure}[h]
\center{\includegraphics[scale = 0.65]{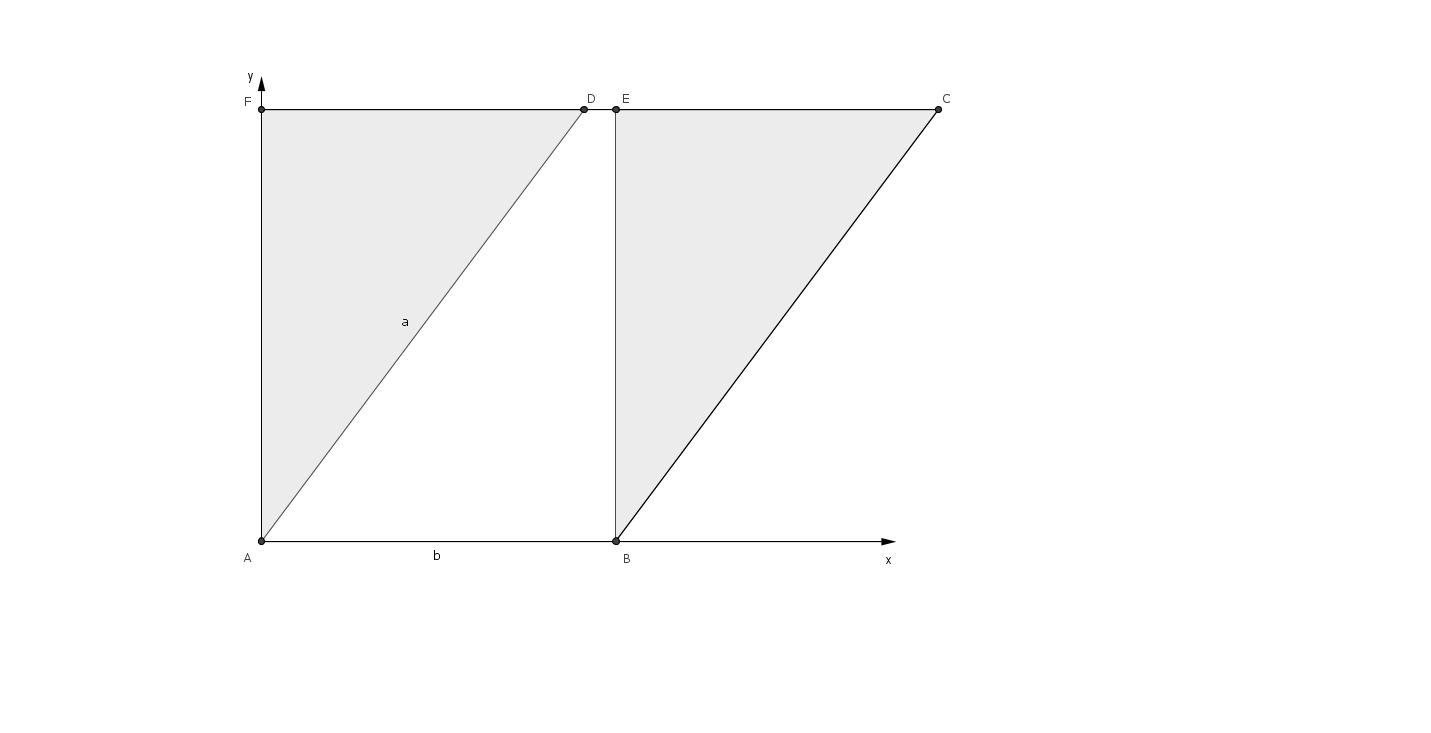}}
\caption{Mutual arrangement of the parallelogram and the rectangle}
\label{fig:coorsys}
\end{figure}

 Neumann eigenvalues and eigenfunctions of rectangle $ABEF$ could be computed explicitly,
$\mu_2 = \dfrac{\pi^2}{a^2}$ and one can take the eigenfunction $u_2 = \cos \dfrac{\pi y}{a} $.
This particular second Neumann eigenfunction doesn't depend on $x$.

The translation by vector $\overrightarrow{AB}$ maps triangle $FAD$ to the triangle $EBC$. So the function $u_2$ resticted to the triangle $FAD$ could be extended to triangle
$EBC$ as $\widetilde{u_2}(x, y) = u_2(x-a, y) $. Define the function
$$
	\widehat{u_2}(x, y) =
		\begin{cases}
			u_2(x, y), & \text{if the point $(x, y)$ belongs to the quadrangle $ABED$;} \\
			\widetilde{u_2}(x,y), & \text{if the point $(x, y)$ belongs to the triangle $BEC$.}
		\end{cases}
$$
The fact that $u_2$ doesn't depend on $x$ provides that $\widehat{u_2} \in C^1$. Rayleigh quotient of $\widehat{u_2}$ is equal exactly to the second Neumann eigenvalue
of the rectangle ABEF:
\begin{eqnarray}
R[\widehat{u_2}] = \dfrac{\displaystyle \iint\limits_{ABCD} | \nabla \widehat{u_2}(x, y) |^2 \diff x \! \diff y}
									{\displaystyle \iint\limits_{ABCD} \widehat{u_2}(x, y)^2 \diff x \! \diff y}
 = \dfrac{\displaystyle \iint\limits_{ABED} | \nabla u_2(x, y) |^2 \diff x \! \diff y + \iint\limits_{BEC} | \nabla \widetilde{u_2}(x, y) |^2 \diff x \! \diff y}
{\displaystyle \iint\limits_{ABED} u_2(x, y)^2 \diff x \! \diff y + \displaystyle \iint\limits_{BEC} \widetilde{u}(x, y)^2  \diff x \! \diff y} =  \nonumber \\
 = \dfrac{\displaystyle \iint\limits_{ABED} | \nabla u_2(x, y) |^2 \diff x \! \diff y + \iint\limits_{ADF} | \nabla u_2(x, y) |^2 \diff x \! \diff y }
{\displaystyle \iint\limits_{ABED} u_2(x, y)^2 \diff x \! \diff y + \displaystyle \iint\limits_{ADF} u_2(x, y)^2 \diff x \! \diff y } = 
\dfrac{\displaystyle \iint\limits_{ABEF} | \nabla u_2(x, y) |^2 \diff x \! \diff y }
					{\displaystyle \iint\limits_{ABEF} u_2(x, y)^2 \diff x \! \diff y } = \mu_2(ABEF). \nonumber
\end{eqnarray}
By the variational principle for the laplacian eigenvalues this means that $\mu_2(ABCD) \leq \mu_2(ABEF)$.
This is the statement of the theorem besides the additional condition on the projection of a longer side.

In fact, this additional condition is not necessary in the general case. It's enough to provide the cutting of the rectangle into finite number of pieces and translate
these pieces by vectors lying on $x$ axis to form the parallelogram. Then the same inequality as above is easy to show.

\begin{figure}[H]
\center{\includegraphics[scale = 0.55]{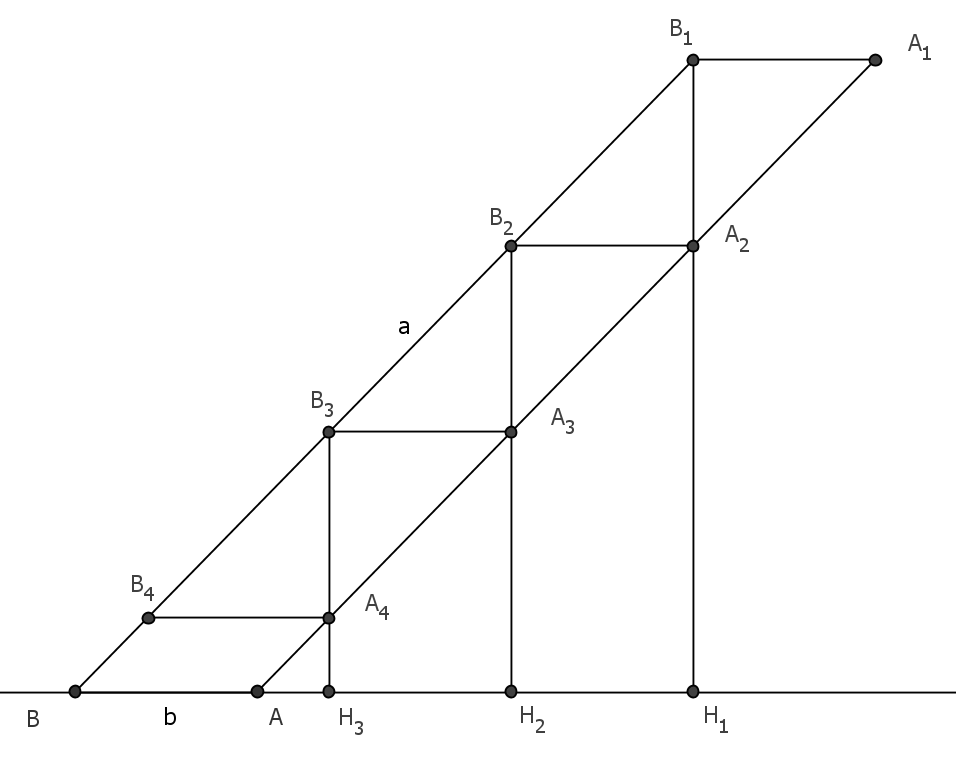}}
\caption{Parallelogram cutting}
\label{fig:manycuts}
\end{figure}

\autoref{fig:manycuts} illustrates an example of the cutting. Let's describe the construction in a general case.
Denote by $A$, $B$ the vertices on the shorter side and by $A_1$, $B_1$ the vertices on the other shorter side such that $A_1A$ and $B_1B$ are the longer sides, 
$B$ and $A_1$ are the vertices with the smallest angles. Let $H_1$ be the foot of the perpendicular from $B_1$ to the line $AB$.
If the altitude $B_1H_1$ falls inside the parallelogram, then it's just the previous case. Otherwise, denote the intersection point with the side $AA_1$ by $A_2$.
Run a parallel line to the shorter side $AB$ through the point $A_2$. Let the intersection of this line with the other longer side be the point $B_2$.
Then drop another perpendicular from $B_2$ to the shorter side $AB$.
If it falls on the side, then it would be a stop. Otherwise, let us call the intersection with the longer side $AA_1$ by $A_3$ and repeat the previous steps.
It's clear that after one iteration the height decreases by $b \sin \alpha$.
That means that after a finite number of steps the height of $AA_iB_iB$ would become less than $b \sin \alpha$.
And the altitude from a point $B_i$ to the line $AB$ lies inside the parallelogram after the finite number of iterations.

\begin{figure}[H]
\center{\includegraphics[scale = 0.7]{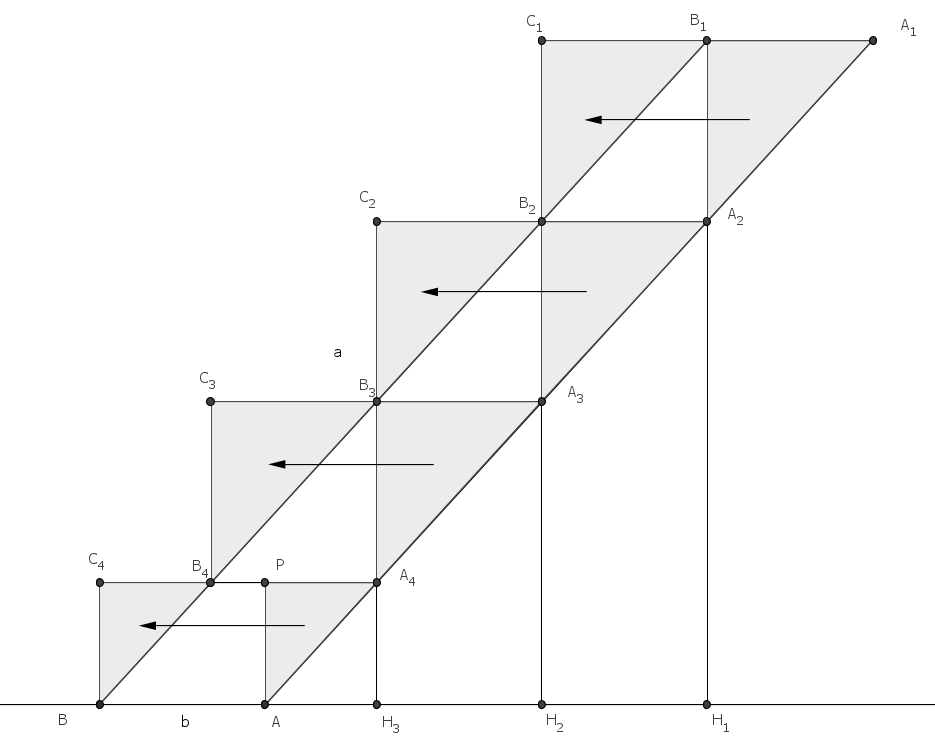}}
\caption{Translation of the pieces}
\label{fig:translate}
\end{figure}

Now the parallelogram is split into a finite number of smaller parallelograms by horizontal parallel lines $A_iB_i$.
All these smaller parallelograms have equal angles and the bases because the sides $A_iB_i$ are all equal and parallel to the side $AB$.
All the heights except the height of the last parallelogram are equal to $b \sin \alpha$. Last parallelogram height is in general less then $b \sin \alpha$
(besides the case when the last altitude base $H_i$ coinsides with the point $A$).
Cut every equal parallelogram part into two triangles by the altitudes $B_iA_{i+1}$. Cut the last parallelogram into
two parts by the perpendicular to the shorter side running through the point~$A$. Translate all the triangles that have a side containing in the $AA_1$
by vector $\overrightarrow{AB}$. This translation is shown in \autoref{fig:translate}.
Each parallelogram $A_{i+1}A_{i}B_{i}B_{i+1}$ is mapped onto rectangle with the same base. Then translate all the rectangles horizontally to form a rectangle with
the same base as the initial parallelogram. It's easy to see that such transformation has an inverse one and it also translates parts by
horizontal vectors. Consider the first Neumann eigenfunction on the rectangle. It does not depend on the horizontal coordinate.
The inverse transformation translate this eigenfunction to a function on the parallelogram. The Rayleigh quotient of the function on the
parallelogram is the same as on the rectanguar because there is no difference either to integrate the function over the domain or to integrate the translated function 
over the translated domain. This is the same argument as above.
So, we get the inequality in this case too.

\end{proof}

\begin{Theorem}{}
Consider a parallelogram $Q$ with the smallest angle $\alpha$ and the lengths of the sides $a \geq b$.
If $\dfrac{b}{a} \leq 2 \sin \alpha - 1$, then $M_2(Q) \leq M_2(S)$, where $S$ is a square.

\end{Theorem}

\begin{proof}
A direct calculation provides that $M_2(S) = (4a)^2 \dfrac{\pi^2}{a^2} = 16 \pi^2$.
Since $\sin \alpha \leq 1$, one has $2 \sin \alpha - 1 \leq \sin \alpha $. So the assumption of the previous theorem is true. This implies that
$$M_2(Q) = \mu_2(Q) p^2(Q) \leq \mu_2(R) p^2(Q) = 4 \dfrac{\pi^2}{(a \sin \alpha)^2} (a + b)^2 = 4 \left(1 + \dfrac{b}{a} \right)^2
													\left(\dfrac{\pi}{\sin \alpha}\right)^2.$$

Rewrite $\dfrac{b}{a} \leq 2 \sin \alpha - 1$ as $1 + \dfrac{b}{a} \leq 2 \sin \alpha$. Thus,
$$M_2(Q) \leq 4 (2 \sin \alpha)^2 \left(\dfrac{\pi}{\sin \alpha}\right)^2 \leq 16 \pi ^2 = M_2(S).$$

\end{proof}

\begin{remark}
 The inequality $2 \sin \alpha - 1 \geq 0$ holds only for parallelograms with the smallest angle that is not less than $\dfrac{\pi}{6}$.
\end{remark}

The proof of the theorem is inspired by the famous Wallace -- Bolyai -- Gerwien theorem \cite{Wallace}, \cite{Bolyai}, \cite{Gerwien}.
Two polygons or polyhedra are called scissors congruent if and only if one could be cut into finitely many pieces and rearranged
through translations and rotations to form the other one.
The theorem states that if and only if two polygons have the same area, then they are scissors-congruent.
It's well-known that this theorem doesn't hold for three-dimensional polyhedra. The Dehn's invariant is an algebraical characteristic of polyhedron and it is
the same for scissors congruent polyhedra, but there exist polyhedra od the same volume and different Dehn's invariant. Sydler has proved that
if $P$ and $Q$ both have the same volume and the same Dehn invariant, then $P$ and $Q$ are scissors congruent \cite{Sydler}.
It's interesting to investigate the connection between isospectrality and the Dehn's invariant. Weyl's law provides that isospectral polyhedra have the same volume 
\cite{Weyl}. It's natural to state the problem.

\newtheorem{Prb}{Problem}
\begin{Prb}
Let there be two polyhedra with the same Dirichlet (Neumann) spectra. Is it necessary that the Dehn's invariants are equal for these polyhedra?
\end{Prb}

We can state this problem as a question: "Can one hear the Dehn's invariant?".

\end{document}